\documentclass[12pt, reqno]{amsart} 
\usepackage{amsmath, amsthm, amscd, amsfonts, amssymb, graphicx, color}
\usepackage[bookmarksnumbered, plainpages]{hyperref}

\numberwithin{equation}{section}
\theoremstyle{plain}
\newtheorem{theorem}{Theorem}[section]
\newtheorem{lemma}[theorem]{Lemma}

\theoremstyle{definition}

\newtheorem{example}[theorem]{Example}
\theoremstyle{remark}

\newcommand{\n}{\zeta}
\newcommand{\te}{\theta}

\begin{document}
\setcounter{page}{1}

\title[Oscillation conditions of  half-linear delay difference equations]{Oscillation conditions of  half-linear delay difference equations of second-order}

\author[C. Jayakumar]{C. Jayakumar}

\address{C. Jayakumari\newline Department of Mathematics, Mahendra Arts \& Science College (Autonomous), Kalipatti 637501, Tamil Nadu, India}
\email{\textcolor[rgb]{0.00,0.00,0.84}{jaya100maths@gmail.com}}
\author[A. Murugesan]{A. Murugesan}

\address{A. Murugesan \newline Department of Mathematics, Government Arts College (Autonomous), Salem  636007, Tamil Nadu, India}

\email{\textcolor[rgb]{0.00,0.00,0.84}{amurugesan3@gmail.com}}

\author[V. Govindan]{Vediyappan Govindan}

\address{Vediyappan Govindan \newline Department of Mathematics, Hindustan Institute of Technology and Science, Padur 603103, Tamil Nadu,  India}

\email{\textcolor[rgb]{0.00,0.00,0.84}{govindoviya@gmail.com}}



\author[C. Park]{Choonkil Park$^*$}

\address{Choonkil Park\newline Department of Mathematics, Research Institute for Convergence of Basic Science, Hanyang University, Seoul 04763, Korea}
\email{\textcolor[rgb]{0.00,0.00,0.84}{baak@hanyang.ac.kr}}


\keywords{oscillation;  non-oscillation;  delay;  half-linear; advanced type; second-order difference equation\\ $^*$Corresponding authors: C. Park (email: baak@hanyang.ac.kr, fax: +82-2-2281-0019, orcid: 0000-0001-6329-8228).}

\begin{abstract}
The goal of our work is to investigate  the oscillation and asymptotic properties of a class of difference equations
\begin{equation}\label{eqstar}
\Delta (r(\n)(\Delta x(\n))^{\alpha})+ q(\n) x^{\alpha}(\n-\sigma)=0; \quad \n \geq \n_o \tag{$*$}
\end{equation}
and
\begin{equation}\label{eqstars}
\Delta (r(\n)(\Delta x(\n))^{\alpha})+ q(\n) x^{\alpha}(\n-\sigma+1)=0; \quad \n \geq \n_o \tag{$**$}
\end{equation}
with the condition $\sum^{\infty} \frac{1}{r^{\frac{1}{\alpha}}(\n)}< \infty$. In contrast to most previous studies, the oscillation of the investigated equation is obtained with only one condition.
\end{abstract}
\maketitle

\baselineskip=16pt


\section{Introduction}

We investigate the oscillation properties of a class of second-order half-linear difference equations
\begin{equation}\label{e1.1}
\Delta (r(\n)(\Delta x(\n))^{\alpha})+ q(\n) x^{\alpha}(\n-\sigma)=0, \quad \n \geq \n_o,
\end{equation}
and
\begin{equation}\label{e1.1'}
\Delta (r(\n)(\Delta x(\n))^{\alpha})+ q(\n) x^{\alpha}(\n-\sigma+1)=0, \quad \n \geq \n_o, \tag{$1.1'$}
\end{equation}
where the difference operator (forward) $\Delta$ is given by $\Delta x(\n)= x(\n+1)-x(\n)$, $\sigma \geq 0$ is an integer for \eqref{e1.1} and $\sigma \geq 1$ for \eqref{e1.1'},  and $\alpha=\frac{m}{n}$ for relatively prime odd positive integers $m$ and $n$. Assume that the real sequences $\{r(\n)\}^{\infty}_{\n=\n_0}$ and $\{q(\n)\}^{\infty}_{\n=\n_0}$ meet the following requirements:
\begin{itemize}
\item[$(H_1)$] $\{r(\n)\}^{\infty}_{\n=\n_0}$ is a positive sequence;
\item[$(H_2)$] $q(\n) \geq 0$ and $q(\n) \not\equiv 0$ for infinitely many values of $\n \in \mathbb{N} (\n_0)=\{\n_0, \n_0+1, \n_0+2, \cdots\}$.
\end{itemize}

``A solution of \eqref{e1.1} is a non-trivial real sequence $\{x(\n)\}$ that is defined for $\n \geq \n_0-\sigma$ and satisfies \eqref{e1.1} for all $\n \geq \n_0$. Obviously, if $x(\n)=A_\n$ for $\n=\n_0-\sigma$, $\n_0-\sigma+1, \ldots, \n_0-1$ are given, then \eqref{e1.1} has a unique solution satisfying the above initial conditions. \eqref{e1.1} is said to be oscillatory if a solution $\{x(\n)\}$ of \eqref{e1.1} is neither eventually positive nor eventually negative, and if it is not, it is said to be non-oscillatory. If all the solutions of \eqref{e1.1} are non-oscillatory, then it is called non-oscillatory''\cite{21}.

The equation \eqref{e1.1} is said to be in canonical form if
\begin{equation}\label{e1.2}
R(\n):= \sum_{s=\n_0}^{\n-1} \frac{1}{r^{\frac{1}{\alpha}}(s)} \rightarrow \infty \quad as \quad \n \rightarrow \infty.
\end{equation}

The equation \eqref{e1.1} is in non-canonical form if
\begin{equation}\label{e1.3}
\te(\n):= \sum_{s=\n}^{\infty} \frac{1}{r^{\frac{1}{\alpha}}(s)} < \infty.
\end{equation}

In recent years, the oscillation conditions for difference equations have been examined for first and higher orders (see \cite{4, 5,   22, 9,  13, 14, 18, le,  21}). We recommend \cite{2, 1, 3, 15} and the nearly 500 sources listed therein for a general theory of oscillation of difference equations. Second-order difference equations have considered far less attention in the study than first-order difference equations, despite the fact that they occur in mathematical biology, economics, and other fields (see  \cite{5}). Many papers \cite{9, 11, 13, 14, 18}  have some recent results on second-order difference equations. On the other hand, it looks that much less is known about the oscillation properties of \eqref{e1.1}. Most of the papers \cite{5, bl, 10, 11, 12, 16,  26, zr} deal with the qualitative properties of solutions to differential and difference equations.

The structure of non-oscillatory solutions differ tremendously between non-canonical and canonical difference equations. Generally, the first order difference of a positive solution $\{x(\n)\}$ of the equation \eqref{e1.1} is of constant sign eventually, but Eq. \eqref{e1.2} assures that the solution is  eventually increasing. The equation \eqref{e1.1} has been analyzed in canonical form a number of times.

The goal of our study is to analyze the oscillation condition and asymptotic properties of the equation \eqref{e1.1} in the form of non-canonical. As a result, we'll assume that \eqref{e1.3} holds in the next sections. We review at a few key oscillation outcomes for second-order non-canonical equations that inspired us to do this research.

In \cite{29},  Zhang {\it et al.}  obtained oscillation conditions for the following equation
\begin{equation*}
\Delta ( p(\n) (\Delta x(\n))^\alpha) + q(\n+1) f(x(\n+1)) = 0, \quad \n=0,1,2,\ldots.
\end{equation*}

In \cite{19}, Li examined and established oscillatory criteria for the solutions of the following equation
\begin{equation*}
\Delta (p(\n) g(\Delta x(\n))) + q(\n+1) f(x(\n+1)) = 0, \quad \n =0,1,2,\ldots.
\end{equation*}

Saker \cite{28} investigated the following equation
\begin{equation}\label{e1.6}
\Delta (p(\n)\Delta x(\n)) + q(\n) f(x(\n-\sigma)) = 0, \quad \n = 0,1,2,3,\ldots.
\end{equation}
and established conditions for oscillation of (\ref{e1.6}).

Dinakar {\it et al.} \cite{6} established sufficient conditions   for the oscillation of the following equation
\begin{equation}\label{e1.7}
\Delta (a(\n) (\Delta x(\n)^{\alpha})) + q(\n) x^{\alpha}(\sigma (\n))=0, \quad \n \geq \n_0,
\end{equation}
and proved that every solutions of \eqref{e1.7} is oscillatory if
\begin{equation*}
\sum_{\n=\n_0}^{\infty} \left( \frac{1}{a(\n)} \sum_{s=\n_0}^{\n-1} A^{\alpha} (\sigma (s)) q(s) \right)^{\frac{1}{\alpha}}=\infty,
\end{equation*}
where $A(\n)=\sum_{s=\n}^{\infty} \frac{1}{a^{\frac{1}{\alpha}}(s)}$.

In \cite{9}, Grace and Alzabut established new oscillations criteria for non-linear second-order equation
\begin{equation*}
\Delta (a(\n)(\Delta y(\n))^{\alpha})+q(\n) x^{\gamma} (\n-m+1)+ c(\n) x^{\mu}(\n+m^{*}+1)=0,
\end{equation*}
by comparison with the oscillation results of first order difference equation.

In \cite{21}, Murugesan and  Jayakuma derived  oscillatory criteria for the difference equation in non-canonical form of advanced type
\begin{equation*}
\Delta (r(\n) (\Delta x(\n))^{\alpha}) + q(\n) x^{\alpha} (\n+\sigma)=0, \quad \n \geq \n_o
\end{equation*}

For the purpose of simplicity and generality, we can only deal with eventually positive solutions of \eqref{e1.1}. Any functional inequality is assumed to hold  eventually.

\section{Main results}

\begin{theorem}\label{thm2.1}
If
\begin{equation}\label{e2.1}
\sum_{\n=\n_0}^{\infty} \left( \frac{1}{r(\n)} \sum_{s=\n_0}^{\n-1} q(s) \right)^{\frac{1}{\alpha}} = \infty,
\end{equation}
then every solution $\{x(\n)\}$ of \eqref{e1.1} oscillates or $\lim_{\n \rightarrow \infty} x(\n)=0$.
\end{theorem}

\begin{proof}
Assume to the contrary that \eqref{e1.1} has an eventually positive solution $\{x(\n)\}$.  Then either $\Delta x(\n) > 0 $ or $\Delta x(\n) < 0 $ holds eventually, say,   $\n \geq \n_1$ and $\n_1 \geq \n_0$.

If $\Delta x(\n) < 0$, then we have $\lim_{\n \rightarrow \infty}x(\n)=c \geq 0$ exists. Let $c > 0$. Summing \eqref{e1.1} from $\n_1$ to $\n-1$, we have 
\begin{equation*}
r(\n) (\Delta x(\n))^{\alpha} - r(\n_1) (\Delta x(\n_1))^{\alpha} + \sum_{s=\n_1}^{\n-1} q(s) x^{\alpha} (s-\sigma)=0,
\end{equation*}
which gives that
\begin{equation*}
 r(\n) (\Delta x(\n))^{\alpha} \le - \sum_{s=\n_1}^{\n-1} q(s) x^{\alpha} (s-\sigma)
\end{equation*}
and hence
\begin{equation*}
\Delta x(\n) \le - \frac{c}{r^{\frac{1}{\alpha}}(\n)} \left( \sum_{s=\n_1}^{\n-1} q(s) \right)^{\frac{1}{\alpha}}.
\end{equation*}
Summing the above from $\n_1$ to $\n-1$, we  have
\begin{equation*}
x(\n)-x(\n_1) \le -c \sum_{u=\n_1}^{\n-1} \frac{1}{r^{\frac{1}{\alpha}}(u)} \left( \sum_{s=\n_1}^{u-1} q(s) \right)^\frac{1}{\alpha}.
\end{equation*}
Taking $\n \rightarrow \infty$ and using \eqref{e2.1}, we see that $x(\n) \rightarrow -\infty$ as $\n \rightarrow \infty$. This is a contradiction and hence $c=0$.

Suppose that $\Delta x(\n)>0$.  Let us define a sequence $\{w(\n)\}$ given by
\begin{equation*}
w(\n):= \frac{r(\n) (\Delta x(\n))^{\alpha}}{x^{\alpha} (\n-\sigma)} > 0.
\end{equation*}
Clearly, $\{w(\n)\}$ satisfies
\begin{equation}\label{e2.2}
\Delta w(\n)= -q(\n)-  \frac{ \alpha w(\n+1) \Delta x(\n-\sigma)}{x(\n-\sigma)} \le -q(\n).
\end{equation}
Summing \eqref{e2.2} from $\n_1$ to $\n-1$, we obtain
\begin{equation*}
w(\n) \le w(\n_1)- \sum_{s=\n_1}^{\n-1} q(s).
\end{equation*}
In view of \eqref{e2.1} and \eqref{e1.3}, we get that $\{\sum_{s=\n_1}^{\n-1} q(s)\}$ is an unbounded sequence, which is a contradiction. Hence the case $\Delta x(\n)>0$ is not possible and hence we complete  the proof.
\end{proof}

\begin{theorem}\label{thm2.2}
If one of the following conditions
\begin{equation}\label{e2.3}
\sum_{s=\n_0}^{\infty} \left( \frac{1}{r(\n)} \sum_{s=\n_1}^{\n-1} q(s) \te^{\alpha} (s-\sigma) \right)^{\frac{1}{\alpha}}= \infty
\end{equation}
and 
\begin{equation}\label{e2.4}
\sum_{s=\n_0}^{\infty} q(s) \te^{\alpha+1}(s+1) = \infty
\end{equation}
holds, then \eqref{e1.1} oscillates.
\end{theorem}

\begin{proof}
On the contrary, let us suppose that \eqref{e1.1} has an eventually positive solution $\{x(\n)\}$. Then, either $\Delta x(\n) >0$ or $\Delta x(\n) <0$ eventually,  $\n \geq \n_1$, $\n_1 \geq \n_0$.

\noindent \textit{Case I:}

 Assume that $\Delta x(\n)<0$ and \eqref{e2.3} holds. Now, for $l>\n$, we have
\begin{align*}
x(l)-x(\n) &= \sum_{s=\n}^{l-1} \Delta x(s)\\
          &\le \left( r(\n)( \Delta x(\n))^{\alpha} \right)^{\frac{1}{\alpha}} \sum_{s=\n}^{l-1} \frac{1}{r^{ \frac{1}{\alpha}}(s)}.
\end{align*}
Thus, as $\n \rightarrow \infty$, we obtain
\begin{equation}\label{e2.5}
-x(\n) \le \left( r(\n_1)(\Delta x(\n_1))^{\alpha} \right)^{\frac{1}{\alpha}} \te(\n).
\end{equation}
Using \eqref{e2.5} in \eqref{e1.1}, we have
\begin{equation*}
\Delta (r(\n)(\Delta x(\n))^{\alpha}) \le L^{\alpha} q(s) \te^{\alpha} (\n-\sigma),
\end{equation*}
where $L= \left( r(\n_1) (\Delta x(\n_1))^{\alpha} \right)^{\frac{1}{\alpha}} <0.$
Summing  the above from $\n_1$ to $\n-1$, we obtain
\begin{align*}
r(\n)(\Delta x(\n_1))^{\alpha} &\le r(\n_1) \Delta x(\n_1)+L^{\alpha} \sum_{s=\n_1}^{\n-1} q(s)\te^{\alpha}(s-\sigma)\\
&\le L^{\alpha} \sum_{s=\n_1}^{\n-1} q(s)\te^{\alpha} (s-\sigma).
\end{align*}
It follows from the last inequality that 
\begin{equation*}
x(\n)-x(\n_1) \le L \left( \sum_{u=\n_1}^{\n-1} \frac{1}{r(u)} \sum_{s=\n_1}^{u-1} q(s) \te^{\alpha} (s-\sigma)\right)^{\frac{1}{\alpha}}.
\end{equation*}
By using \eqref{e2.3}, we see that $\lim_{\n \rightarrow} x(\n)= - \infty$ which is a contradiction on the assumption of $\{x(\n)\}$.

\noindent \textit{Case II:}
Assume that $\Delta x(\n)<0$ and \eqref{e2.4} holds.
Set
\begin{equation}\label{e2.8}
u(\n)= \frac{r(\n)(\Delta x(\n))^{\alpha}}{x^{\alpha}(\n-\sigma)}.
\end{equation}
Obviously, $u(\n)<0$, $\n \geq \n_1$. Since $\{r(\n)(\Delta x(\n))^{\alpha}\}$ is nonincreasing, we obtain 
\begin{equation*}
r(s)(\Delta x(s))^{\alpha} \le r(\n) (\Delta x(\n))^{\alpha},\quad s \geq \n \geq \n_1.
\end{equation*}
Summing  the above from $\n$ to $l-1$, we have 
\begin{equation*}
x(l) \le x(\n) + r^{\frac{1}{\alpha}} (\n) \Delta x(\n) \sum_{s=\n}^{l-1} \frac{1}{r^{\frac{1}{\alpha}}(s)}.
\end{equation*}
Letting $l \rightarrow \infty$ in the above inequality, we have
\begin{equation*}
0 \le x(\n) + r^{\frac{1}{\alpha}} (\n) \Delta x(\n) \te(\n), \quad \n \geq \n_1.
\end{equation*}
So we get
\begin{equation*}
\frac{r^{\frac{1}{\alpha}} (\n) \Delta x(\n)}{x(\n)} \te(\n) \geq -1,\quad \n \geq \n_1
\end{equation*}
or
\begin{equation*}
-1 \le \te(\n) u^{\frac{1}{\alpha}}(\n) <0, \quad \n \geq \n-1.
\end{equation*}
From \eqref{e2.8}, we have
\begin{equation}\label{e2.11}
\Delta u(\n)+ q(\n)+ \frac{ \alpha u^{\frac{1+\alpha}{\alpha}} (\n)}{r^{\frac{1}{\alpha}}(\n-\sigma)} \le 0, \quad \n \geq \n_1.
\end{equation}
Multiplying  \eqref{e2.11} by $\te^{\alpha+1} (\n)$ and summing from $\n_1$ to $l-1$, we have 
\begin{multline}\label{e2.12}
\te^{\alpha+1}(l) u(l)- \te^{\alpha+1} (\n_1) u(\n_1) + \sum_{s=\n_1}^{l-1} q(s) \te^{\alpha+1} (s+1)\\
 - (\alpha+1) \sum_{s=\n_1}^{l-1} \frac{u(s) \te^{\alpha} (s)}{r^{\frac{1}{\alpha}} (s)} + \alpha \sum_{s=\n_1}^{l-1} \frac{\te^{\alpha+1} (s+1) u^{\frac{\alpha+1}{\alpha}}(s)}{r^{\frac{1}{\alpha}}(s-\sigma)} \le 0.
\end{multline}
Now
\begin{equation*}
\left| \sum_{s=\n_1}^{\infty} \frac{\te^{\alpha} (s) u(s)}{r^{\frac{1}{\alpha}} (s)} \right| \le \sum_{s=\n_1}^{\infty} \frac{1}{r^{\frac{1}{\alpha}}(s)} |\te^{\alpha} (s) u(s)|
 \le \sum_{s=\n_1}^{\infty} \frac{1}{r^{\frac{1}{\alpha}} (s)} < \infty
\end{equation*}
and
\begin{equation*}
\left| \sum_{s=\n_1}^{\infty} \frac{\te^{\alpha+1}(s+1) u^{\frac{\alpha+1}{\alpha}} (s)}{r^{\frac{1}{\alpha}} (s-\sigma)} \right|  \le \sum_{s=\n_1}^{\infty} \frac{1}{r^{\frac{1}{\alpha}}(s-\sigma)} \left| \te^{\alpha}(s) u(s) \right|^{\frac{\alpha+1}{\alpha}} < \sum_{s=\n_1}^{\infty} \frac{1}{r^{\frac{1}{\alpha}}(s-\sigma)} < \infty.
\end{equation*}
By applying the above inequalities, letting $l \rightarrow \infty$ in \eqref{e2.12}, we see that
\begin{equation*}
\sum_{s=\n_1}^{\infty} q(s) \te^{\alpha+1}(s+1) < \infty,
\end{equation*}
which contradicts  \eqref{e2.4}.

So in either case $\Delta x(\n)< 0$ is impossible.

If $\Delta x(\n) >0$, then we repeat the same procedure as in the proof of Theorem \ref{thm2.1} and so we  omitt the proof.
\end{proof}

\begin{lemma} \label{lem2.1} \cite{6}
Suppose that
\begin{equation}\label{e2.15}
\sum_{\n=\n_0}^{\infty} q(\n)= \infty.
\end{equation}
If $\{x(\n)\}^{\infty}_{\n=\n_0}$ is an eventually positive solution of \eqref{e1.1}, then $\Delta x(\n) < 0$.
\end{lemma}

\begin{theorem}\label{thm2.3}
If
\begin{equation}\label{e2.16}
\limsup_{\n \rightarrow \infty} \left( \te^{\alpha}(\n) \sum_{s=\n_1}^{\n-1} q(s) \right) >1,
\end{equation}
then \eqref{e1.1} oscillates.
\end{theorem}

\begin{proof}
Let us assume on the contrary, without loss of generality, that \eqref{e1.1} has an eventually positive solution $\{x(\n)\}$. Then, either $\Delta x(\n) <0$ or $\Delta x(\n) >0$ eventually, for $\n \geq \n_1$, $\n_1 \geq \n_0$.
First now, we suppose that $\Delta x(\n) <0$. Since $ \{r^{\frac{1}{\alpha}}(\n) \Delta x(\n) \}$ is decreasing, we get
\begin{align}\label{e2.17}
x(\n-\sigma) \geq x(\n) \geq &\sum_{s=\n}^{\infty} \frac{1}{r^{\frac{1}{\alpha}}(s)} \left( -r^{\frac{1}{\alpha}}(s) \Delta x(s) \right)\nonumber\\ &\geq -\te(\n) r^{\frac{1}{\alpha}} (\n) \Delta x(\n) \geq 0.
\end{align}
Summing \eqref{e1.1} from $\n_1$ to $\n-1$ and using \eqref{e2.17}, we obtain
\begin{align*}
-r(\n) (\Delta x(\n))^{\alpha} \geq &\sum_{s=\n_1}^{\n-1} q(s) x^{\alpha} (s-\sigma)\nonumber\\ &\geq x^{\alpha} (\n-\sigma) \sum_{s=\n_1}^{\n-1} q(s) \nonumber\\ &\geq -r(\n) (\Delta x(\n))^{\alpha} \te^{\alpha}(\n) \sum_{s=\n_1}^{\n-1} q(s).
\end{align*}
Taking the limsup,  we have  a contradiction to  \eqref{e2.16}.

Now suppose that $\Delta x(\n)>0$. Then by \eqref{e1.3} and \eqref{e2.16}, we have \eqref{e2.15} holds. Then by  Lemma \ref{lem2.3}, we have a contradiction and the proof is now complete.
\end{proof}

The comparison with the second-order linear delay difference equation in canonical form
\begin{equation}\label{e2.21}
\Delta (\widetilde r(\n) \Delta x(\n-1))+\widetilde q(\n) x(\n-\sigma)=0
\end{equation}
yields the following oscillation criterion for \eqref{e1.1'}.

\begin{lemma} \label{lem2.3} \cite{22} 
If the difference inequality
\begin{equation*}
\Delta (\widetilde r(\n) \Delta x(\n-1))+\widetilde q(\n) x(\n-\sigma) \le 0
\end{equation*}
has an eventually positive solution, then Eq. \eqref{e2.21} has an eventually positive solution.
\end{lemma}

\begin{lemma}\label{lem2.2}
If $\alpha \geq 1$ and  $\{x(\n)\}$ is an eventually positive solution of \eqref{e1.1'}, then
\begin{equation} \label{e2.19}
\left( \frac{r^{\frac{1}{\alpha}} (\n) \Delta x(\n)}{x(\n-\sigma+1)} \right)^{\alpha-1} \le \te^{1-\alpha} (\n).
\end{equation}
\end{lemma}

\begin{proof}
Suppose that  \eqref{e1.1'} has an eventually positive solution $\{x(\n)\}$. Then we get  $\Delta x(\n)>0$ or $\Delta x(\n)<0$ eventually for $\n \geq \n_1$, $\n_1 \geq \n_0$.

Let us assume first that $\Delta x(\n)<0$. Since $\alpha \geq 1$, repeatting  the procedure as in the proof of  Theorem \ref{thm2.3}, we obtain
\begin{equation*}
x(\n-\sigma+1) \geq -\te(\n) r^{\frac{1}{\alpha}}(\n) \Delta x(\n) \geq 0,
\end{equation*}
which implies \eqref{e2.19}.

Assume now that $\Delta x(\n) >0$. Note that $R(\n-\sigma)+\te(\n-\sigma)=\te(\n_0) >0$ together with \eqref{e1.3} implies $R(\n-\sigma) \geq \te(\n)$. Now
\begin{align*}
 x(\n-\sigma+1) \geq x(\n-\sigma) &\geq \sum_{s=\n_1}^{\n-\sigma-1} \Delta x(s)\nonumber \\
&=\sum_{s=\n_1}^{\n-\sigma-1}\frac{1}{r^{\frac{1}{\alpha}}(s)} r^{\frac{1}{\alpha}}(s) \Delta x(s)\nonumber\\
&\geq R(\n-\sigma) r^{\frac{1}{\alpha}}(\n) \Delta x(\n)\nonumber \\
&\geq \te(\n) r^{\frac{1}{\alpha}}(\n) \Delta x(\n),
\end{align*}
which is obviously the same as \eqref{e2.19} and this completes the proof.
\end{proof}

\begin{theorem}\label{thm2.5}
Let $\alpha \geq 1$. Suppose that Eq. \eqref{e2.21} with
\begin{equation*}
\widetilde r(\n) = \te(\n) \te(\n+1) r^{\frac{1}{\alpha}}(\n),
\end{equation*}
\begin{equation*}
\widetilde q(\n)= \frac{1}{\alpha} \te(\n+1) \te^{\alpha-1}(\n) \te(\n-\sigma+1) q(\n)
\end{equation*}
oscillates. Then \eqref{e1.1'} oscillates.
\end{theorem}

\begin{proof}
Let us assume on the contrary that $\{x(\n)\}$ is  an eventually positive solution of \eqref{e1.1'}. Now the equation \eqref{e1.1'} can be rewritten as
\begin{equation*}
\Delta \left( r^{\frac{1}{\alpha}}(\n) \Delta x(\n) \right) + \frac{1}{\alpha} \left( r^{\frac{1}{\alpha}}(\n) \Delta x(\n) \right)^{1-\alpha} q(\n) x^{\alpha} (\n-\sigma+1) \le 0.
\end{equation*}
By Lemma \ref{lem2.2}, we conclude that $\{x(\n)\}$ is a solution of the second-order inequality
\begin{equation*}
\Delta \left( r^{\frac{1}{\alpha}}(\n) \Delta x(\n) \right) + \frac{1}{\alpha} \te^{\alpha-1} (\n) q(\n) x(\n-\sigma+1) \le 0,
\end{equation*}
for $\n \geq \n_1$, $\n \geq \n_0$. Now the inequality can be modified into
\begin{multline}\label{e2.25}
\frac{1}{\te(\n+1)} \Delta \left( \te(\n) \te(\n+1) r^{\frac{1}{\alpha}} (\n) \Delta \left( \frac{x(\n)}{\te(\n)} \right) \right)+ \frac{1}{\alpha} \te^{\alpha-1}(\n) q(\n) x(\n-\sigma+1) \le 0.
\end{multline}
Put
\begin{equation*}
u(\n-1)= \frac{x(\n)}{\te(\n)}.
\end{equation*}
Then the inequality \eqref{e2.25} reduces to
\begin{equation}\label{e2.27}
\Delta (\widetilde r(\n) \Delta u(\n-1))+\widetilde q(\n) u(\n-\sigma) \le 0.
\end{equation}
This shows that $\{u(\n)\}$ is an eventually positive solution of \eqref{e2.27}. Then by Lemma \ref{lem2.3}, \eqref{e2.21} has an eventually positive solution which is a contradiction and hence the proof is completed.
\end{proof}

\begin{example}
Consider  the difference equation
\begin{equation}\label{e2.28}
\Delta \left( 2^{\frac{\n}{3}} (\Delta x(\n))^{\frac{1}{3}} \right) + \lambda_0 2^{\n} x^{\frac{1}{3}}(\n-1)=0; \quad \n \geq 1, 
\end{equation}
where $r(\n)=2^{\frac{\n}{3}}$, $q(\n) = \lambda_0 2^{\n}$, $\alpha = \frac{1}{3}$ and $\sigma = 1$.
We can easily verify that \eqref{e2.1} holds and by Theorem \ref{thm2.1}, every eventually non-oscillatory solution of \eqref{e2.28} converges to zero as $\n \rightarrow \infty$. Also, we can easily verify that \eqref{e2.16} holds and hence by Theorem \ref{thm2.3}, \eqref{e2.28} oscillates if $\lambda_0>1$.
\end{example}

\begin{example}
Consider  the difference equation
\begin{equation}\label{e2.29}
\Delta \left( (\n(\n-1) )^{\frac{1}{3}} (\Delta x(\n))^{\frac{1}{3}} \right) + \n^{\frac{4}{3}} x^{\frac{1}{3}} (\n-1) =0; \quad \n \geq 1,
\end{equation}
where  $r(\n)=(\n(\n-1))^{\frac{1}{3}}$, $q(\n)= \n^{\frac{4}{3}}$, $\alpha={\frac{1}{3}}$ and $\sigma=1$.
We can easily show that $\te(\n)= \frac{1}{\n-1}$ and also
\begin{equation*}
\sum_{s=1}^{\infty} q(s) \te^{\alpha+1}(s+1) = \infty.
\end{equation*}
Hence \eqref{e2.29} oscillates by  Theorem \ref{e2.2}.
\end{example}

\begin{example}
Consider  the difference equation
\begin{equation}\label{e2.30}
\Delta \left( (\n(\n+1))^{\frac{5}{3}} (\Delta x(\n))^{\frac{5}{3}} \right) + \frac{4(\n^{2}-1) \n^{\frac{2}{3}}}{3} x(\n-1)=0; \quad \n=1,2,3,\ldots,
\end{equation}
where  $r(\n)= (\n(\n+1))^{\frac{5}{3}}$, $q(\n)= \frac{4(\n^{2}-1) \n^{\frac{2}{3}}}{3}$, $\alpha= \frac{5}{3}$ and $\sigma=2$.
We can easily calculate that $\te(\n)=\frac{1}{\n}$,
\begin{equation}\label{e2.31}
 \widetilde r(\n)= \te(\n) \te(\n+1) r^{\frac{1}{\alpha}}(\n)=1,
\end{equation}
and
\begin{equation}\label{e2.32}
\widetilde q(\n)=4.
\end{equation}
Clearly, $\sum_{\n=1}^{\infty} \widetilde q(\n) = \infty$. Then by \cite[Theorem 3.1]{27}, every solution of \eqref{e2.21} is oscillatory with $\sigma = 2$, \eqref{e2.31} and \eqref{e2.32}. One such a solution of \eqref{e2.21} is $\{ (-1)^{\n} \}^{\infty}_{\n=1}$. By Theorem \ref{thm2.5}, \eqref{e2.30} oscillates.
\end{example}

\section{Conclusion}

The results of this study are unique and also have a high degree of generality. In contrast to most previous studies, our oscillation result of the investigated equation \eqref{e1.1} is obtained with only one condition. Also, we have derived oscillation criteria for \eqref{e1.1'} by comparison with the second-order linear delay difference equation in canonical form. To illustrate our results, three examples are provided.

\medskip

\section*{Declarations}

\medskip

\noindent \textbf{Availablity of data and materials}\newline
\noindent Not applicable.

\medskip

\noindent \textbf{Human and animal rights}\newline
\noindent We would like to mention that this article does not contain any studies
with animals and does not involve any studies over human being.

\medskip

\noindent \textbf{Conflict of interest}\newline
\noindent The authors declare that they have no competing interests.

\medskip

\noindent \textbf{Fundings} \newline
\noindent The authors declare that there is no funding available for this paper.

\medskip

\noindent \textbf{Authors' contributions}\newline
\noindent The authors equally conceived of the study, participated in its
design and coordination, drafted the manuscript, participated in the
sequence alignment, and read and approved the final manuscript.

\end{document}